\documentclass[8pt]{article}
\usepackage{indentfirst,latexsym,bm}
\usepackage{amsfonts}
\usepackage{amssymb}
\usepackage{times}
\usepackage[leqno]{amsmath}
\usepackage{dsfont}
\usepackage{amsthm}
\usepackage[all]{xy}
\usepackage{hyperref}
\usepackage{color,xcolor}
\begin{document}
\title{Group-theoretical property of some integral non-degenerate fusion categories}
\author{Zhiqiang Yu}
\date{}
\maketitle

\newtheorem{theo}{Theorem}[section]
\newtheorem{prop}[theo]{Proposition}
\newtheorem{lemm}[theo]{Lemma}
\newtheorem{coro}[theo]{Corollary}
\theoremstyle{definition}
\newtheorem{conj}[theo]{Conjecture}
\newtheorem{defi}[theo]{Definition}
\newtheorem{exam}[theo]{Example}
\newtheorem{ques}[theo]{Question}
\newtheorem{rema}[theo]{Remark}

\newcommand{\A }{\mathcal{A}}
\newcommand{\C }{\mathcal{C}}
\newcommand{\B }{\mathcal{B}}
\newcommand{\D }{\mathcal{D}}
\newcommand{\E }{\mathcal{E}}
\newcommand{\K}{\mathds{k}}
\newcommand{\I }{\mathcal{I}}
\newcommand{\M }{\mathcal{M}}
\newcommand{\Q }{\mathcal{O}}
\newcommand{\Rep}{\text{Rep}}
\newcommand{\Y }{\mathcal{Z}}
\newcommand{\Z }{\mathbb{Z}}

\abstract
We show that an integral non-degenerate fusion category $\C$ is
group-theoretical if the Frobenius-Perron dimensions of its simple objects are either 1 or powers of a prime $p$.
\section{Introduction}
Throughout this paper, we work over an algebraically closed field $\K$  of characteristic zero, we denote $\K^*:=\K\backslash {\{0}\}$.

A fusion category is a semisimple $\K$-linear finite abelian tensor category. A fusion category $\C$ is   a pointed fusion category  if it is tensor equivalent to $\text{Vec}_G^\omega$, the category of $G$-graded finite-dimensional vector spaces over $\K$,  where $G$ is a finite group and $\omega\in Z^3(G,\K^*)$ is a $3$-cocycle, we note that the $3$-cocycle $\omega$ gives the associativity constraint of $\text{Vec}_G^\omega$. Recall that a fusion category $\C$ is said to be group-theoretical if $\C$ is (categorically) Morita equivalent to a pointed fusion category $\text{Vec}_G^\omega$ via an indecomposable $\C$-module category, see \cite{EGNO,ENO3} for the  precise definitions of module category and  Morita equivalence of fusion categories. Equivalently, a fusion category $\C$ is   group-theoretical if and only if there exists a braided tensor equivalence between their Drinfeld centers $\Y(\C)\cong\Y(\text{Vec}_G^\omega)$ \cite[Theorem 1.3]{ENO3}.

Since the property of being group-theoretical for fusion categories is invariant under taking Drinfeld centers and fusion subcategories,  fusion subcategories of $\Y(\text{Vec}_G^\omega)$ are always group-theoretical. In \cite{NNW}, the authors studied explicitly fusion subcategories of $\Y(\text{Vec}_G^\omega)$ in terms of subgroups of $G$, characters  and cocycles. Moreover, it was shown that a braided fusion category $\C$ is group-theoretical if and only if $\C$ is tensor equivalent to an equivariantization of some pointed fusion category \cite[Theorem 7.2]{NNW}, for the definition of equivariantization of fusion category see \cite{DrGNO2} for details.

A braided fusion category $\C$ is a fusion category   equipped with a braiding
\begin{align*}
c_{X,Y}:X\otimes Y\overset{\sim}{\to} Y\otimes X, ~\forall X,Y\in\C,
 \end{align*}
which satisfies the hexagon axioms, see \cite{EGNO}. For any fusion subcategory $\D$ of a braided fusion category $\C$, the centralizer $\D_\C'$ of $\D$ in $\C$ is generated by objects $X\in\C$ that centralize every object of $\D$. That is,
 $\D_\C'$  is generated by objects  $X\in\C$   which satisfy
\begin{align*}
c_{Y,X}c_{X,Y}=\text{id}_{X\otimes Y},~\forall Y\in\Q(\D),
\end{align*}
 where $\Q(\D)$ is the set of equivalence classes of simple objects of $\D$. And we call the fusion subcategory $\C_\C'$ the M\"{u}ger center of $\C$ and it will be denoted by $\C'$. We say a braided fusion category $\C$ is non-degenerate if its M\"{u}ger center $\C'=\text{Vec}$, where $\text{Vec}$ is the category of finite-dimensional vector spaces over $\K$.

It was proved in \cite[Theorem 1.2]{Na3} that an integral non-degenerate fusion category $\C$ is group-theoretical  if all simple objects of $\C$ have Frobenius-Perron dimensions $1$ or $2$.
In this paper, we generalize the  conclusion  of  \cite[Theorem 1.2]{Na3} to an arbitrary prime $p$. More precisely, let $p$ be a prime,  if all simple objects of an integral non-degenerate fusion category $\C$   have Frobenius-Perron dimensions $1$ or powers of $p$, then   $\C$ is   group-theoretical, see Theorem \ref{maintheorem} and Corollary \ref{CoroPowerP}.  The proofs of Theorem \ref{maintheorem} and Corollary \ref{CoroPowerP} are based on   \cite{CGPW,ENO2,Kir,Michler,Na3}, it is also  an analogue  of that of \cite[Theorem 4.5]{OY}.

The organization of this paper is  as follows. In section \ref{Preliminaries}, we recall some basic properties of fusion categories and braided fusion categories that we use throughout. In section \ref{mainresult},   we prove our main result (Theorem \ref{maintheorem} and Corollary \ref{CoroPowerP}) about the group-theoretical property of some integral non-degenerate fusion categories.

\section{Preliminaries} \label{Preliminaries}
In this section, we will recall some important definitions and properties about fusion categories and braided fusion categories, we refer the reader to \cite{DrGNO2,EGNO,ENO1,ENO3,GN} for more details.

Given a $\K$-linear finite semisimple abelian category $\C$,   we denote by $\Q(\C)$ the set of isomorphism classes of simple objects of $\C$, the  cardinal of $\Q(\C)$ is called the rank of $\C$ and it will be  denoted by $\text{rank}(\C)$.

Let $\C$ be a fusion category, then there exists a unique ring homomorphism  FPdim(-)  from the Grothendieck ring $\text{Gr}(\C)$ to $\K$ such that $\text{FPdim}(X)$ is a positive algebraic integer for all objects $X\in\Q(\C)$ \cite[Theorem 8.6]{ENO1}, and $\text{FPdim}(X)$ is called the Frobenius-Perron dimension of object $X$. Moreover, we define the Frobenius-Perron dimension $\text{FPdim}(\C)$ of a fusion category $\C$  by
\begin{align*}
\text{FPdim}(\C):=\sum_{X\in\Q(\C)}\text{FPdim}(X)^2.
\end{align*}
A fusion category $\C$ is integral if $\text{FPdim}(X)\in\Z$, $\forall X\in\Q(\C)$.

A simple object $X$ of a fusion category $\C$ is invertible if $X\otimes X^*=X^*\otimes X=I$, the unit object of $\C$, if and only if $\text{FPdim}(X)=1$. It is easy to see that the set of isomorphism classes of invertible objects of $\C$ form a finite group, which will be  denoted by $G(\C)$ below. A fusion category $\C$ is pointed if it is tensor equivalent to the fusion category $\text{Vec}_G^\omega$ of $G$-graded finite-dimensional vector spaces over $\K$ \cite{ENO1}, where $G$ is a finite group and $\omega\in Z^3(G,\K^*)$ is a $3$-cocycle. We denote by $\C_\text{pt}$ the maximal pointed fusion subcategory of any fusion category $\C$ below, that is, $\C_\text{pt}$ is the fusion subcategory generated by invertible objects of $\C$, so $G(\C)=\Q(\C_\text{pt})$.

Let $G$ be a finite group.  $\C$ is a $G$-graded fusion category, if $\C=\oplus_{g\in G}\C_g$ is a direct sum of abelian subcategories, and for arbitrary elements $g,h\in G$ we have that \begin{align*}
\C_g\otimes\C_h\subseteq\C_{gh}, ~(\C_g)^*\subseteq\C_{g^{-1}}.
 \end{align*}
 By definition, the trivial component $\C_e$ is a fusion subcategory of $\C$.
The $G$-grading of a fusion category $\C=\oplus_{g\in G}\C_g$ is faithful  if $\C_g\neq0$ for all $g\in G$, in this special situation,  we also say that $\C$ is a $G$-extension of $\C_e$. Given a faithfully $G$-graded fusion category $\C=\oplus_{g\in G}\C_g$,
\begin{align*}
\text{FPdim}(\C_g)=\text{FPdim}(\C_e)~ \text{ for any $g\in G$}
\end{align*}
and  $\text{FPdim}(\C)=|G|\text{FPdim}(\C_e)$ by \cite[Proposition 8.20]{ENO1}, where we denote
\begin{align*}
\text{FPdim}(\C_g):= \sum_{X\in\Q(\C_g)}\text{FPdim}(X)^2.
\end{align*}

Let  $\C_\text{ad}$ be the  adjoint fusion subcategory of $\C$, that is, $\C_\text{ad}$  is  generated by simple objects $Y$ such that $Y\subseteq X\otimes X^*$ for some simple object $X\in\C$. And, for any fusion category $\C$, there is a faithful grading $\C=\oplus_{g\in U_\C}\C_g$  with $\C_\text{ad}$  being the trivial component \cite[Corollary 3.7]{GN}, this grading is called the universal grading of $\C$, we denote by $U_\C$  the universal grading group of $\C$. Moreover, for any other faithful grading of $\C=\oplus_{g\in G}\C_g$, there is a surjective group homomorphism from $U_\C$ to $G$.

A fusion category $\C$ is nilpotent, if there exists  a natural number $n$ such that $\C^{(n)}=\text{Vec}$, where $\C^{(0)}:=\C$, $\C^{(1)}:=\C_\text{ad}$, $\C^{(m)}:=(\C^{(m-1)})_\text{ad}$, for all integers $m\geq1$ \cite{GN}. Equivalently, $\C$ is nilpotent if there is a sequence of fusion subcategories
$\text{Vec}=\C_0\subseteq\C_1\subseteq\cdots\subseteq\C_n=\C$ and finite groups $G_i$ such that $\C_i$ is obtained as a $G_i$-extension of $\C_{i-1}$, $1\leq i\leq n$. If these finite groups $G_i$ ($1\leq i\leq n$) are cyclic groups, then we say that $\C$ is a cyclically nilpotent fusion category, see \cite{ENO3}.

By definition, for any pointed fusion category $\C$, $\C_\text{ad}=\text{Vec}$, so pointed fusion categories are  always nilpotent; and    \cite [Theorem 8.28]{ENO1} shows that fusion categories of prime power Frobenius-Perron dimensions are cyclically nilpotent.  A fusion category $\C$ is weakly group-theoretical if it is (categorically) Morita equivalent to a   nilpotent fusion category,   furthermore, a fusion category $\C$ is called  solvable if  it is  Morita equivalent to a cyclically nilpotent fusion category, see \cite{EGNO,ENO3}.

Let $\C$ be a  braided fusion category with a braiding $c$, we denote by $\C^\text{rev}$ the braided  fusion category with reverse braiding $c^\text{rev}$, where $\C^\text{rev}:=\C$ as fusion category, and
\begin{align*}
 c^\text{rev}_{X,Y}:=c_{Y,X}^{-1}, ~\forall X,Y\in\C^\text{rev}.
  \end{align*}

A braided fusion category  $\C$ is symmetric   if $\C'=\C$; in addition, a symmetric fusion category $\C$ is Tannakian if $\C\cong \text{Rep}(G)$, where $G$ is a finite group, and the braiding $c$ of $\text{Rep}(G)$ is given as the ordinary reflection of linear vector spaces, i.e.,
\begin{align*}
c_{X,Y}(x\otimes y)=y\otimes x,   ~x\in X, y\in Y,~\forall X,Y\in\Q(\text{Rep}(G)).
\end{align*}
When $\C'=\text{Vec}$, we call  $\C$  a non-degenerate fusion category \cite{DrGNO2}. A braided fusion category $\C$ is slightly degenerate if $\C'=\text{sVec}$, the symmetric fusion category of finite-dimensional super-vector spaces over $\K$.

Let $G$ be a finite group. Let  $\C$ be a  fusion category with $G$ acted on $\C$ by tensor auto-equivalences. Then there is a well-define fusion category $\C^G$, which is called the equivariantization of $\C$ by $G$ \cite{DrGNO2,EGNO}. Objects of $\C^G$ are the   $G$-equivariant objects $X$ of $\C$, that is, there is a natural isomorphism  $\mu_g:g(X)\overset{\sim}{\to} X$ satisfying  the following equation
\begin{align*}
\mu_{gh}\circ \nu_{g,h}^X=\mu_g\circ g(\mu_h)
\end{align*}
where $\nu_{g,h}$ is the natural isomorphism associated to the $G$-action on $\C$. Moreover, it is shown that $\text{FPdim}(\C^G)=\text{FPdim}(\C)|G|$ \cite[Proposition 4.26]{DrGNO2}, and the simple objects and fusion rules of $\C^G$ were studied explicitly in \cite{BuNa}.

Given a braided fusion category $\C$, let $\E=\text{Rep}(G)$ be a maximal Tannakian subcategory of $\C$. Then the de-equivariantization $\C_G$ admits a $G$-grading structure, making $\C_G$  a $G$-crossed braided  fusion category \cite{DrGNO2,EGNO,Kir}. Moreover, the  trivial component $\C_G^0$ of the $G$-grading of $\C_G$ is a  braided fusion subcategory, and $\E'_G:=(\E_\C')_G\cong\C_G^0$ as  braided fusion category \cite[Proposition 4.56]{DrGNO2}.  The braided fusion category   $\C_G^0$ is called the core of $\C$ \cite[Definition 5.8]{DrGNO2},  it was shown that the core of a braided fusion category $\C$ is independent of the choice of the maximal Tannakian subcategory $\E$ \cite[Theorem 5.9]{DrGNO2}.
See \cite[Sections 4 $\&$ 5]{DrGNO2} and the references therein for details.

Conversely, given a   braided  fusion category $\D$ with a specific group homomorphism $\rho: G\to \text{Aut}_\otimes^\text{br}(\D)$, where $\text{Aut}_\otimes^\text{br}(\D)$ is the group of braided auto-equivalences of $\D$. In order to  obtain a braided fusion category $\C$  such that $\C_G^0\cong \D$ as braided fusion category, one   need to show the existence of a $G$-crossed braided fusion category $\A$ such that  $\A_0\cong\D$, where $\A_0$ is the trivial component of the $G$-grading of $\A$ by \cite[Theorem 7.12]{ENO2}.  By using the  Brauer-Picard groups and Picard groups of fusion categories developed in \cite{ENO2}, there are two obstructions $o_3(\rho)$ and $o_4(\rho)$, which are characterized explicitly using cohomological data, to the existence of a desired braided  $G$-crossed braided fusion category $\A$.  Therefore, one have to show the vanishing of these two obstructions $o_3(\rho)$ and $o_4(\rho)$, see \cite[Sections 7 $\&$ 8]{ENO2}  for a more complete description.

In addition, if $\C$ is  a non-degenerate fusion category, then the above inverse procedure of taking the core of $\C$ is called gauging, which is well-known and useful in physics, see  \cite{CGPW,Mu} and the references therein, for example.

A finite abelian group $G$ is called a metric group $(G,\eta)$ if  there is  a non-degenerate quadratic form $\eta$  on $G$. More precisely, there is a morphism $\eta:G\to\K^*$ satisfying  $\eta(g^{-1})=\eta(g)$, and $\eta$ defines a non-degenerate bi-character $\beta$ on $G$, where
\begin{align*}
\beta(g,h):=\frac{\eta(gh)}{\eta(g)\eta(h)}, ~\forall g,h\in G.
 \end{align*}
 It is well-known that there is a bijective correspondence between pointed non-degenerate fusion categories  and metric groups,
 see \cite[subsection 2.11]{DrGNO2}. In this paper, we will use $\C(G,\eta)$ to denote the non-degenerate pointed fusion category determined by the metric group $(G,\eta)$.
\section{Main result}\label{mainresult}
In this section, we always assume  $p$  is a prime. We prove that an integral non-degenerate fusion category $\C$ is  group-theoretical  if all simple objects of $\C$ have   Frobenius-Perron dimensions $1$ or  $p$.

We begin with the following proposition. The proof   of the proposition is easy, we include it for the reader's convenience.
\begin{prop}\label{tannakiandim}Let $\C$ be an integral non-degenerate but not pointed fusion category. If all simple objects of $\C$ have Frobenius-Perron dimensions $1$ or  powers of $p$, then  $p|\text{FPdim}(\E)$, where $\E$ is an arbitrary maximal Tannakian subcategory of $\C$.
\end{prop}
\begin{proof}We first show that $\C$ contain a Tannakian category $\text{Rep}(\Z_p)$. Since $\C$ is not pointed, $\C_\text{ad}\ncong \text{Vec}$. For any non-invertible simple object $X$ of $\C$, $X\otimes X^*\in\C_\text{ad}$ by definition, notice that \begin{align*}
X\otimes X^*=\oplus_{g\in G(X)}g\oplus \oplus_Z N_{X,X^*}^ZZ,
 \end{align*}
 where $G(X)=\{g|g\otimes X=X, g\in G(\C)\}$ is a subgroup of $G(\C_\text{ad})$ and $Z$ ranges over the set of isomorphism classes of non-invertible simple objects of $\C_\text{ad}$.  Since non-invertible simple objects of $\C$ have Frobenius-Perron dimensions   power of $p$,
\begin{align*}
\text{FPdim}(X)^2=|G(X)|+pb, \text{where $b$ is a non-negative integer,}
\end{align*}
we see $G(X)$ is non-trivial $p$-group, so  $p$ divides $\text{FPdim}((\C_\text{ad})_\text{pt})$.

Assume that $\text{FPdim}((\C_\text{ad})_\text{pt})=p^mr$, since $(\C_\text{ad})_\text{pt}$  is a symmetric fusion subcategory by  \cite[Corollary 6.8]{GN}, $(\C_\text{ad})_\text{pt}$ contains a non-trivial Tannakian category if $p$  is odd or $m\geq2$ \cite[Corollary 2.50]{DrGNO2}.
On the contrary, assume that  $m=1$ and $p=2$, and that $(\C_\text{ad})_\text{pt}$ contains a symmetric fusion category $\B:=\text{sVec}=\langle g\rangle$, where the simple object $g$ generates  $\B$.
Then   $\B_\C'$  is a slightly degenerate fusion subcategory of $\C$ \cite[Corollary 3.11]{DrGNO2}. Moreover, $\B_\C'$ is not pointed, otherwise, $\B_\C'\cong\text{sVec}\boxtimes\D$ \cite[Proposition 2.6]{ENO3}, where $\D$ is a pointed non-degenerate fusion category. We  then arrive to a contradiction, because $\C\cong\D\boxtimes\D_\C'$ with $\D_\C'$ being an integral fusion category of Frobenius-Perron dimension $4$, so $\C$ must be pointed.

By assumption,   $g\subseteq X\otimes X^*$ for all non-invertible simple objects $X$ of $\B_\C'$, that is, $g$ must fixes all   non-invertible simple objects of $\B_\C'$, but this contradicts to \cite[Lemma 5.4]{Mu}. Thus,  $\C$ contains a Tannakian subcategory $\E_1\cong \text{Rep}(\Z_p)$.
Let $\E$ be an arbitrary maximal Tannakian subcategory of $\C$ such that $\E_1\subseteq\E$, then   $p$ divides $\text{FPdim}(\E)$ \cite[Proposition 8.15]{ENO1}. Therefore, $p$ divides the Frobenius-Perron dimension of any maximal Tannakian subcategory of $\C$, as maximal Tannakian subcategory of   $\C$ have same the Frobenius-Perron dimension \cite[Theorem 5.9]{DrGNO2}.
\end{proof}
Given two fusion subcategories $\A,\B$ of a  fusion category $\C$, we denote by  $\A\vee\B$ the fusion subcategory of $\C$ generated by $\A$ and $\B$. More precisely, $\A\vee\B$ is the smallest fusion subcategory of $\C$ that contains  both $\A$ and $\B$.

Let  $G$ be a finite group. Given a faithfully $G$-graded fusion category $\C=\oplus_{g\in G}\C_g$, assume that the trivial component $\C_e$ is pointed and $G(\C_e)=N$. Then the $G$-grading structure of $\C$ induces a natural group action of $N$ on the set $\Q(\C_g)$ for all $g\in G$, i.e., $(h,X)\mapsto h\otimes X$  for all $h\in N$ and $X\in\Q(\C_g)$. One can easily prove that the group $N$ acts transitively on the sets $\Q(\C_g)$ for all $g\in G$, see  \cite[Lemma 4.1]{OY} for example. In particular, the simple objects of the component $\C_g$ all have same the Frobenius-Perron dimension, $\forall g\in G$.

We first give our main  theorem in a  special situation. The proof of the following theorem is similar to  that of \cite[Theorem 4.2]{OY}.  Recall that a braided fusion category $\C$ is weakly anisotropic if $\C$ does not contain a non-trivial Tannakian fusion subcategory that is invariant under all braided auto-equivalences of $\C$ \cite[Definition 5.16]{DrGNO2}.
\begin{theo}\label{tannpowerp}Let  $\C$ be an integral non-degenerate fusion category such that $\text{FPdim}(X)$ is a power  of $p$   for all objects $ X\in\Q(\C)$. If all Tannakian subcategories of $\C$  have Frobenius-Perron dimensions powers of $p$, then $\C$ is  nilpotent and  group-theoretical.
\end{theo}
\begin{proof}We assume that $\C$ is not pointed below, otherwise there is nothing to prove.  By Proposition \ref{tannakiandim},  $\C$ contains a non-trivial Tannakian subcategory. Indeed, the existence of non-trivial Tannakian subcategory  can also be deduced from  \cite[Theorem 3.1]{Na3}, as $\C$ is always a solvable fusion category \cite[Theorem 7.2]{Na2}, by
\cite[Theorem 3.1]{Na3} $\C$ contains a non-trivial Tannakian subcategory.

Let $\E=\text{Rep}(G)$ be a maximal Tannakian subcategory of $\C$, and $\D:=\C_G=\oplus_{g\in G}\D_g$ with $\D_e=\C_G^0$ being the trivial component of the $G$-grading. Then the core $\C_G^0$ of $\C$ is a weakly anisotropic fusion category by \cite[Corollary 5.19]{DrGNO2},  hence \cite[Theorem 1.1]{Na3} says  that $\C_G^0$ is a pointed non-degenerate fusion category. Notice that the Frobenius-Perron dimensions of  simple objects of $\D$ also belong to  the set ${\{p^i|i\geq0}\}$ by \cite[Proposition 4.26]{DrGNO2}.

Assume $\text{FPdim}(\C)=p^mq^{n_1}_1\cdots q^{n_s}_s$, where $q_j\neq p$ are distinct primes and $n_j$ are positive integers for all $j$, $1\leq j\leq s$. Since $|G|=p^t$ for some positive integer $t$ by assumption,
we deduce from \cite[Proposition 4.56]{DrGNO2} that $\text{FPdim}(\C_G^0)=p^{m-2t}q^{n_1}_1\cdots q^{n_s}_s$. Therefore, we have a braided tensor equivalence of non-degenerate pointed fusion categories
\begin{align*}
\C_G^0\cong\C(N,\varphi)\boxtimes\C(H_1,\eta_1)\boxtimes\cdots\boxtimes\C(H_s,\eta_s),
\end{align*}
where $(N,\varphi)$ and $(H_j,\eta_j)$ are metric groups of orders $p^{m-2t}$ and $q^{n_j}_j$, $1\leq  j\leq s$, respectively. Clearly, as $G$ acts on $\C(N,\varphi)$ and $\C(H_j,\eta_j)$ as braided auto-equivalences, these fusion subcategories   are stable under  the action of $G$, therefore we obtain braided fusion subcategories $\C(N,\varphi)^G$ and $\C(H_j,\eta_j)^G$ of $\E_\C'\cong(\C_G^0)^G$, hence the generated fusion category $\C(N,\varphi)^G\vee\C(H_1,\eta_1)^G\vee\cdots\vee\C(H_s,\eta_s)^G$ is a braided fusion subcategories of $\E_\C'$.

By \cite[Proposition 4.26]{DrGNO2}, for all $1\leq j\leq  s$, we have \begin{align*}
\text{FPdim}(\C(N,\varphi)^G)=|G||N|,~\text{FPdim}(\C(H_j,\eta_j)^G)=|G||H_j|.
 \end{align*}
By \cite[Proposition 8.15]{ENO1}, it is easy to see  that   $\C(N,\varphi)^G\vee\C(H_1,\eta_1)^G\vee\cdots\vee\C(H_s,\eta_s)^G$ and $\E_\C'$  have the same Frobenius-Perron dimensions, so
\begin{align*}
\E_\C'=\C(N,\varphi)^G\vee\C(H_1,\eta_1)^G\vee\cdots\vee\C(H_s,\eta_s)^G.
\end{align*}

The universal grading property   implies that $\D_\text{ad}\subseteq\D_e$  \cite[Corollary 3.10]{GN}, so $\D$ is nilpotent. Then for an arbitrary element $g\in G$,    \cite[Lemma 4.1]{OY} says that $q^{n_1}_1\cdots q^{n_s}_s$ divides $\text{rank}(\D_g)$.  Meanwhile for any $g\in G$, $\text{rank}(\D_g)=|\Q(\D_e)^g|$ by
\cite[Lemma 10.7]{Kir}, where $\Q(\D_e)^g\subseteq\Q(\D_e)$ is the set of isomorphism classes of simple objects of $\D_e$ that are fixed by the braided auto-equivalence $g$. Notice that $\Q(\D_e)^g$ is an abelian group since $\D_e$  is a pointed non-degenerate fusion category, thus $G$ acts trivially on the non-degenerate the fusion category $\C(H_j,\eta_j)$ for all $j$. Hence, for all $1\leq j\leq s$, since $(|H_j|,|G|)=1$, we see that
 \begin{align*}
 \C(H_j,\eta_j)^G\cong\C(H_j,\eta_j)\boxtimes \text{Rep}(G)
 \end{align*} as braided fusion categories. Consequently, we have a braided tensor equivalence
\begin{align*}
\C\cong\A\boxtimes\C(H_1,\eta_1)\boxtimes\cdots\boxtimes\C(H_s,\eta_s)
\end{align*}
by \cite[Theorem 3.13]{DrGNO2}, where $\A$ is a non-degenerate fusion category of Frobenius-Perron dimension $p^m$. Then $\C$ is nilpotent and group-theoretical by \cite[Theorem 6.10]{DrGNO1}, as desired.
\end{proof}

In fact, Theorem \ref{tannpowerp} implies the following corollary:
\begin{coro}\label{generalcoro}Let $\C$ be an integral non-degenerate weakly group-theoretical fusion category. Assume that $\E=\text{Rep}(G)$ is  a maximal Tannakian subcategory of $\C$. If there is a prime $q$ satisfies that $(q,|G|)=1$ and $(q, \text{FPdim}(X))=1$ for any  simple object  $X\in\Q(\C)$, then we have a braided equivalence  $\C\cong\C(H,\eta)\boxtimes\C(H,\eta)_\C'$, where $(H,\eta)$ is a metric group of order $q^m$  and $m$ is the maximal non-negative integer such that $q^m$ divides $\text{FPdim}(\C)$.
\end{coro}
\begin{proof}(\textit{sketched}) Let $m$ be the maximal non-negative integer such that $q^m$ divides $\text{FPdim}(\C)$. If $m=0$, this is trivial; we assume $m>0$ below.

Since $\C$ is an integral weakly group-theoretical fusion category,  we know that $\D:=\C_G^0$ is a pointed non-degenerate fusion category \cite[Theorem 1.1]{Na3}. Note that the Frobenius-Perron dimensions of the simple objects of $\C$ (and also for $\C_G$ by \cite[Corollary 2.2]{BuNa}) are coprime to $q$ and $(q,|G|)=1$ by assumption, $q^m$ divides $\text{FPdim}(\D)$ as $\text{FPdim}(\D)=\frac{\text{FPdim}(\C)}{|G|^2}$ by \cite[Proposition 4.56]{DrGNO2},  $\D\cong\C(H,\eta)\boxtimes\C(H,\eta)_\D'$ as braided fusion category, where $\C(H,\eta)$ is a pointed non-degenerate fusion category of Frobenius-Perron dimension $q^m$. For the $G$-crossed braided fusion category $\C_G$,
 we know that $G$ acts trivially on  $\C(H,\eta)$ as braided auto-equivalences \cite[Lemma 10.7]{Kir},  a  similar  argument of   Theorem \ref{tannpowerp} shows that $\C(H,\eta)^G\cong\C(H,\eta)\boxtimes\Rep(G)$ as braided fusion categories. Meanwhile $\E_\C'\cong\C(H,\eta)^G\vee(\C(H,\eta)_\D')^G$ as braided fusion category,
 then the conclusion follows by \cite[Theorem 3.13]{DrGNO2}.
\end{proof}

We are now in a position to prove  our main theorem about group-theoretical property of integral non-degenerate fusion categories, which strengthens the conclusion  of
 \cite[Theorem 1.2]{Na3}.
\begin{theo}\label{maintheorem}Let $\C$ be a non-degenerate   fusion category. If $\text{FPdim}(X)\in{\{1,p}\}$ for all simple objects $X\in\Q(\C)$, then $\C$ is  a group-theoretical fusion category.
\end{theo}
\begin{proof} We know that $\C$ contains a non-trivial Tannakian subcategory by Proposition \ref{tannakiandim}. Let $\E=\text{Rep}(G)$ be a maximal Tannakian subcategory of $\C$, and $\D:=\C_G=\oplus_{g\in G}\D_g$ with $\D_e=\C_G^0$ being the trivial component. On the one hand, if $\D$ is pointed, then $\C$ is group-theoretical by \cite[Theorem 7.2]{NNW}. On the other hand, if all Tannakian subcategories of $\C$ have Frobenius-Perron dimensions of powers of $p$, then this is  exactly the result of Theorem \ref{tannpowerp}.

We assume that  $\D$ is not pointed and $|G|=p^jt$ below, where $t>1$ is an integer with $(p,t)=1$, we also have  $j\geq1 $ by Proposition \ref{tannakiandim}.  The argument of  Theorem \ref{tannpowerp} says $\D$ is a nilpotent fusion category. Consequently, simple objects of $\D$ also have Frobenius-Perron dimensions $1$ or $p$ by \cite[Corollary 2.2]{BuNa}. Let $\text{FPdim}(\C)=p^mn$ with $(p,n)=1$, and $\D_e=\C(H_1,\eta_1)\boxtimes\C(H_2,\eta_2)$, where $(H_1,\eta_1)$ and $(H_2,\eta_2)$ are metric groups of orders $p^{m-2j}$ and $\frac{n}{t^2}$, respectively. Note that $m-2j\geq2$, since for any non-invertible simple object $X$ of $\D$, we have $X\otimes X^*\in\D_\text{ad}\subseteq\D_e$,  $p^2|\text{FPdim}(\D_\text{ad})|\text{FPdim}(\D_e)$ \cite[Corollary 5.3]{GN}.

Notice that  the proof of \cite[Lemma 6.3]{Na3} (indeed, one just need to replace $2$ by  the prime $p$) shows that the  universal grading group $U_\D$ of $\D$ is  an abelian group. Meanwhile, $\D$ is a $G$-crossed braided fusion category with a faithful $G$-grading  \cite[Proposition 4.56]{DrGNO2}, and it follows from \cite[Corollary 3.10]{GN} that there is a surjective group homomorphism from $U_\D$ to  $G$, thus $G$ is also an abelian group. Assume that $G=G_1\times G_2$, where $|G_1|=p^j$  and $|G_2|=t$, respectively.

As $G$ is an abelian group, we deduce  from \cite[Theorem 1, section 4]{CGPW} that there exist  a $G_2$-crossed braided fusion category $\widetilde{\D_e}$ with trivial component $\D_e$ and a $G_1$-crossed braided fusion category $\widetilde{\B}$ with trivial component  $\B:=\widetilde{\D_e}^{G_2}$
such that we have a braided tensor equivalence $\C\cong \widetilde{\B}^{G_1}$. Since  $\D_e=\C_G^0$ is a non-degenerate fusion category,   $\B$ is also a non-degenerate fusion category by
\cite[Proposition 4.56]{DrGNO2}. In addition, the same argument of Theorem \ref{tannpowerp} implies
\begin{align*}
\B\supseteq \D_e^{G_2}=\C(H_1,\eta_1)^{G_2}\vee\C(H_2,\eta_2)^{G_2},
\end{align*}
since simple objects of $\C$ have Frobenius-Perron dimensions $1$ or $p$, and $\C$ is an equivariantization of $\widetilde{\B}$,   simple objects of  both $\widetilde{\B}$ and $\B$  have Frobenius-Perron dimensions $1$ or $p$ by \cite[Corollary 2.2]{BuNa}. Therefore, \cite[Proposition 4.26]{DrGNO2} and \cite[Lemma 10.7]{Kir} mean that $G_2$ acts trivially on the pointed non-degenerate fusion category $\C(H_1,\eta_1)$. Consequently, $\B\supseteq \C(H_1,\eta_1)\boxtimes \text{Rep}(G_2)$, as orders of $G_2$ and $H_1$ are coprime, and $H^i(G_2,H_1)=0$ for all non-negative integers $i$. So \cite[Theorem 3.13]{DrGNO2} shows that $\B\cong\C(H_1,\eta_1)\boxtimes\A$ for some non-degenerate fusion category $\A$, moreover, we have
\begin{align*}
\text{FPdim}(\A)=\frac{\text{FPdim}(\B)}{|H_1|}=\frac{|G_2|^2\text{FPdim}(\D_e)}{|H_1|}=n.
\end{align*}

We claim that $\A$ is pointed. In fact, note that $p\nmid \text{FPdim}(\A)$ and $\frac{\text{FPdim}(\A)}{\text{FPdim}(X)^2}$ is an algebraic integer for any simple object $X$ of $\A$ \cite[Theorem 2.11]{ENO3}, then the assumption on simple objects of $\C$ implies that $\A$ must  be a pointed fusion category.  Therefore, $\B$ is a pointed non-degenerate fusion category. In the last, we consider the $G_1$-crossed braided fusion category $\widetilde{\B}$, whose trivial component $\B$ is   pointed.  As we have done in Theorem \ref{tannpowerp} or Corollary \ref{generalcoro},
 \cite[Lemma 10.7]{Kir} also shows that $G_1$ also acts trivially on the pointed non-degenerate fusion subcategory $\A$.

Note that $\A^{G_1}\subseteq\B^{G_1}\subseteq\widetilde{\B}^{G_1}\cong\C$,    $\text{FPdim}(\A)=n$ and $(n,|G_1|)=1$. Hence, the previous argument (same as that of Theorem \ref{tannpowerp}) also shows that $\C$ contains $\A^{G_1}\cong\A\boxtimes \text{Rep}(G_1)$ as a braided fusion subcategory, which implies that $\C\cong\A\boxtimes\A_\C'$ as braided fusion category by
\cite[Theorem 3.13]{DrGNO2}. Hence, $\A_\C'$ is a non-degenerate fusion category of Frobenius-Perron dimension $p^m$. Therefore, $\C$ is nilpotent and group-theoretical by \cite[Corollary 6.8]{DrGNO1}.
\end{proof}

Let $G$ be a finite group. Recall that  the Ito-Michler Theorem of finite groups states that $G$ has an abelian normal $p$-subgroup if and only if $p$ can't divide $\dim_\K(V)$ for any irreducible representation $V$ of $G$, see \cite[Theorem 5.4]{Michler} for details. With the help of  the Ito-Michler Theorem, we can   prove the  following corollary  by using the same method of Theorem \ref{maintheorem}, which also strengthens the conclusion of Theorem \ref{maintheorem}.
\begin{coro}\label{CoroPowerP}
Let $p$ be an arbitrary prime. Let $\C$ be an integral non-degenerate fusion category such that  $\text{FPdim}(X)$ is a power of $p$ for    all simple objects $X$ of $\C$. Then $\C$ is a group-theoretical fusion category.
\end{coro}
\begin{proof}(\textit{sketched}) Assume  $\C$ is not a pointed fusion category, otherwise it is trivial. We know that $\C$ contains a maximal non-trivial Tannakian subcategory $\E=\text{Rep}(G)$ by Proposition \ref{tannakiandim}.
If $\C_G$ is pointed, then $\C$ is group-theoretical \cite[Theorem 7.2]{NNW}.
Assume $\D:=\C_G=\oplus_{g\in G}\D_g$ is not pointed  below. Note that $\C$ is always weakly group-theoretical \cite[Proposition 5.3]{Na3}, same as Theorem \ref{tannpowerp}, we obtain that $\D_e$ is a pointed non-degenerate fusion category \cite[Theorem 1.1]{Na3}. As a fusion subcategory of $\C$, the Frobenius-Perron dimensions of simple objects of $\E$ are also powers of $p$, then the Ito-Michler Theorem of finite groups \cite[Theorem 5.4]{Michler} states that $G$ is isomorphic to the semidirect product $G_1\rtimes G_2$, where $G_1$ is an abelian group such that $(|G_1|,p)=1$ and $G_2$ is a $p$-group.

Notice that \cite[Theorem 1, section 4]{CGPW} still works for semidirect product of finite groups, therefore there exist  a $G_1$-crossed braided fusion category $\widetilde{\D_e}$ with trivial component $\D_e$ and a $G_2$-crossed braided fusion category $\widetilde{\B}$ with trivial component  $\B:=\widetilde{\D_e}^{G_1}$
such that we have a braided tensor equivalence $\C\cong \widetilde{\B}^{G_2}$, then the rest of proof is same of  that of Theorem \ref{maintheorem}.
\end{proof}

Let   $p$ be a prime, and let $\C$ be a slightly degenerate fusion category such that $\text{FPdim}(X)\in {\{p^j|j\geq0}\}$ for all simple objects $X\in\Q(\C)$. Then $\C$ is solvable by \cite[Theorem 7.2]{Na2}. In addition, we can assume that $\C$ does not contain any non-trivial non-degenerate fusion subcategories. Otherwise, $\C\cong\D\boxtimes\B$, where $\B$ is non-degenerate and $\D$ is slightly degenerate. Theorem \ref{maintheorem} shows that $\B$ is group-theoretical, hence we can replace $\C$ by $\D$ if $\D$ is not pointed.

A non-degenerate fusion category $\A$ is called a minimal extension of a slightly degenerate fusion category $\C$, if $\C\subseteq\A$ and $\text{FPdim}(\A)=2\text{FPdim}(\C)$ \cite[Conjecture 5.2]{Mu}.

Let $\C$ be a slightly degenerate fusion category such that $\text{FPdim}(X)\in {\{p^j|j\geq0}\}$ for all simple objects $X\in\Q(\C)$, then $\C$ admits a minimal extension $\A$ by
 \cite[Theorem 3.5]{OY}, which can also assumed to be an integral fusion category
by  \cite[Remark 3.6]{OY}. Indeed, one can show that there exists a minimal extension $\A$ of $\C$ such that $\A\subseteq\Y(\C_G)$ \cite[Theorem 3.5]{OY}, where $\Rep(G)$ is an arbitrary  maximal Tannakian subcategory of $\C$.

As a faithful $\Z_2$-extension of $\C$ \cite[Proposition 3.5.3]{EGNO}, $\A$ is also a solvable fusion category by \cite[Proposition 4.5]{ENO3}.
Since $\C_G^0$ is a slightly degenerate integral weakly anisotropic fusion category \cite[Corollary 5.19]{DrGNO2}, $\C_G^0$ must be pointed by \cite[Theorem 1.1]{Na3}, hence  $\text{sVec}\boxtimes\D^\text{rev}\cong\C_G^0\subseteq\A_G^0$, where $\D$ is a pointed non-degenerate fusion category,
\cite[Theorem 1.3]{ENO3} says that
  \begin{align*}
\Y(\text{sVec})\boxtimes\D\boxtimes \D^\text{rev}\cong\Y(\text{sVec})\boxtimes \Y(\D^\text{rev})\cong \Y(\text{sVec}\boxtimes\D^\text{rev})\cong\Y(\C_G)_G^0\cong\D\boxtimes\A_G^0,
  \end{align*}
  thus,  $\A_G^0\cong\D^\text{rev}\boxtimes \Y(\text{sVec})$ as pointed braided fusion category.
\begin{ques}\label{sldegpowerp}
Are $\C$ and $\A$ group-theoretical fusion categories?
\end{ques}

\section*{Acknowledgements}
The author is supported by NSFC (no.12101541), NSF of Jiangsu Province (no.BK20210785), and Natural Science Foundation of Jiangsu Higher Institutions of China (no.21KJB110006). The author thanks the anonymous referee for   numerous suggestions that helped improve this paper substantially.  Part of this paper was written during a visit of the author at University of Oregon supported by China Scholarship Council (no.201806140143). The author is   grateful to   V. Ostrik for  weekly insightful conversations. He is also grateful to   J. Plavnik for inviting him to Indiana University in September 2019, he  appreciates the Department of Mathematics of both University of Oregon and Indiana University  for their warm hospitality.

\author{Zhiqiang Yu\\ \thanks{Email:\,zhiqyumath@yzu.edu.cn}
\\{\small School  of Mathematical Science,  Yangzhou University,
Yangzhou 225002, China}
}

\end{document}